\documentclass[11pt, a4paper]{article}
\title{Two optimization problems for the Loewner energy}
\usepackage[T1]{fontenc}

\usepackage{mathtools}
\usepackage{lmodern}
\usepackage{amsthm,amsmath,amsfonts,amssymb}
\usepackage{graphicx}
\usepackage{enumerate,enumitem}
\usepackage{authblk}
\usepackage{cite,url}
\usepackage[normalem]{ulem}
\usepackage{calrsfs}

\usepackage[margin=1cm]{caption}

\usepackage[activate={true,nocompatibility},final,tracking=true,kerning=true,spacing=true,factor=1100,stretch=20,shrink=20]{microtype}

\microtypecontext{spacing=nonfrench}

\usepackage{hyperref}
\hypersetup{
    colorlinks=true, 
    linkcolor=black, 
    urlcolor=black, 
    citecolor=black,
    linktoc=all 
}

\allowdisplaybreaks

\setlist[enumerate]{topsep = 1ex, leftmargin=.5cm, itemsep= -3pt}
\setlist[itemize]{topsep = 1ex, leftmargin=.5cm, itemsep= -3pt}

\let\OLDthebibliography\thebibliography
\renewcommand\thebibliography[1]{
  \OLDthebibliography{#1}
  \setlength{\parskip}{1pt}
  \setlength{\itemsep}{2pt}
}

\newtheorem{thm}{Theorem}[section]

\newtheorem{lem}[thm]{Lemma}

\theoremstyle{definition} 
\newtheorem{df}[thm]{Definition}
\newtheorem{ex}[thm]{Example}
\newtheorem{qn}[thm]{Question}
\newtheorem{remark}[thm]{Remark}

\def\weld{h}

\usepackage{longtable,booktabs}

 \usepackage[dvipsnames]{xcolor}

 \global\long\def\domain{D}

\global\long\def\bpt{a} 
 \global\long\def\ept{b} 

 \global\long\def\ii{\mathfrak{i}}
\global\long\def\jj{\mathfrak{j}}

\newcommand{\abs}[1]{\left\lvert #1 \right \rvert}

\newcommand{\brac}[1]{\left \langle #1 \right \rangle}

\newcommand{\mc}[1]{\mathcal{#1}}
\newcommand{\m}[1]{\mathbb{#1}}

\renewcommand\Im{\operatorname{Im}}

\def\PSL{\operatorname{PSL}}

\def\QS{\operatorname{QS}}
\def\tr{\operatorname{Tr}}

\def\WP{\operatorname{WP}}

\def\a{\alpha}
\def\b{\beta}
\def\g{\gamma}

\def\d{\delta}

\def\t{\theta}

\def\l{\lambda}

\def\O{\Omega}
\def\vare{\varepsilon}

\def\Chat{\hat{\m{C}}}

 \def\RP{\mathbb{RP}}
 \def\Homeo{\operatorname{Homeo}}
 \def\gr{\operatorname{gr}}
 \def\AdS{\operatorname{AdS}}

 \def\Id{\operatorname{Id}}
 \def\adj{\operatorname{adj}}
 \def\Ker{\operatorname{Ker}}
 \def\Span{\operatorname{Span}}
 \def\Isom{\operatorname{Isom}}

\def\dd{\mathrm{d}}

\newcommand{\ad}[1]{\overline{#1}}

\def\P{\m P}

\def\1{\mathbf{1}}

\author{Yilin Wang\thanks{Funded by the European Union (ERC, RaConTeich, 101116694). Views and opinions expressed are however those of the author(s) only and do not necessarily reflect those of the European Union or the European Research Council Executive Agency. Neither the European Union nor the granting authority can be held responsible for them.}\\
\emph{Institut des Hautes \'Etudes Scientifiques}\\
35 Route de Chartres, 91440 Bures-sur-Yvette, France\\
\protect\url{yilin@ihes.fr}}

\date{\today}
\begin{document}

 \maketitle
\begin{abstract}
A Jordan curve on the Riemann sphere can be encoded by its conformal welding, a circle homeomorphism. The Loewner energy measures how far a Jordan curve is away from being a circle, or equivalently, how far its welding homeomorphism is away from being a M\"obius transformation. We consider two optimizing problems for the Loewner energy, one under the constraint for the curves to pass through $n$ given points on the Riemann sphere, which is the conformal boundary of hyperbolic $3$-space $\mathbb H^3$; the other under the constraint for $n$ given points on the circle to be welded to another $n$ given points of the circle. The latter problem can be viewed as optimizing positive curves on the boundary of AdS$^3$ space passing through $n$ prescribed points.  We observe that the answers to the two problems exhibit interesting symmetries: optimizing the Jordan curve in $\partial_\infty \mathbb H^3$ gives rise to a welding homeomorphism that is the boundary of a pleated plane in AdS$^3$, whereas optimizing the positive curve in $\partial_\infty\!\operatorname{AdS}^3$ gives rise to a Jordan curve that is the boundary of a pleated plane in $\m H^3$.  
\end{abstract}

\section{Introduction}
Conformal welding 
encodes a Jordan curve into a circle homeomorphism. It is a classical subject in geometric function theory to study the correspondence between the analytic properties of the curve and homeomorphism. 
More precisely, let $\g \subset \Chat = \m C \cup \{\infty\}$ be an oriented Jordan curve (in this work, all Jordan curves are on the Riemann sphere). We denote the two connected components of $\Chat \smallsetminus \g$ on the left and right of $\g$ by $\O$ and $\O^*$, respectively. We write 
$\m H^2 : = \{z \in \m C \,|\, \Im(z) > 0\}$ and $\m H^{2,*} : = \{z \in \Chat  \,|\, \Im(z) < 0\}$. 
From the Carath\'eodory theorem, any conformal map $f$ from $\m H^2$ onto $\O$ extends continuously to a homeomorphism of their closures and defines a homeomorphism from 
$\m {RP}^1$ to $\g$.  
Here we have identified  $\RP^1$ with $\partial_\infty \m H^2 = \m R \cup \{\infty\}$ via the map $[x:y] \mapsto x/y$, which can be further identified with the unit circle in $\m C$ by the Cayley map $z \mapsto (z - \ii)/(z + \ii)$. Similarly, any conformal map $g$ from  $\m H^{2,*}$ onto $\O^*$ defines a homeomorphism from 
$\m {RP}^1$ to $\g$. 
The \emph{welding homeomorphism} (or simply the \emph{welding}) of $\g$ is defined as 
\begin{equation}\label{eq:welding}
   \weld_\g := g^{-1} \circ f |_{\RP^1} \in \Homeo_+(\RP^1)
\end{equation}
where we denote by $\Homeo_+(\RP^1)$ the space of orientation preserving homeomorphisms of $\RP^1$.
We note that there is some ambiguity in this definition. We may pre-compose $f$ and $g$  by M\"obius transformations preserving $\RP^1$, namely by elements in 
$$\PSL (2, \m R) : = \left\{ \a = \begin{pmatrix}
    a & b \\ c & d
\end{pmatrix} \,\Bigg|\, a, b , c,d \in \m R, \,ad -bc = 1\right\}_{/ \a \sim - \a}$$
which act on $\m H^2$ and $\m H^{2,*}$ by  \emph{M\"obius transformations} $z \mapsto (az + b)/(cz+d)$ and the induced action on $\RP^1$ is simply the linear map 
\begin{equation}\label{eq:action_RP1} \RP^1 \to \RP^1\quad : \quad
\begin{pmatrix}
    x\\ y
\end{pmatrix} \mapsto  \begin{pmatrix}
    a & b \\ c & d
\end{pmatrix} \begin{pmatrix}
    x\\ y
\end{pmatrix}.    
\end{equation}
Hence, the welding of a Jordan curve should be considered as an equivalence class in 
$$\PSL (2, \m R) \backslash \Homeo_+(\RP^1) /  \PSL (2, \m R)$$ of the equivalence relation
\begin{equation}\label{eq:equivariant_welding}
\weld_\g \sim \b \circ \weld_\g \circ \a^{-1}, \qquad  \forall \a, \b \in \PSL(2,\m R).
\end{equation}

Given $\weld \in \Homeo_+(\RP^1)$, by solving the \emph{conformal welding problem} for $\weld$, we mean finding a Jordan curve $\g$ and corresponding conformal maps $f,g$ such that $\weld_\g = g^{-1} \circ f |_{\RP^1}= \weld$. 
If a solution $(\g; f,g)$ exists, then $(A \circ \g; A \circ f, A \circ g)$ is also a solution, where $A$ is a conformal automorphism of $\Chat$, namely, 
$$A \in \PSL(2, \m C) :  = \left\{ A = \begin{pmatrix}
    a & b \\ c & d
\end{pmatrix} \,\Bigg|\, a, b , c,d \in \m C, \,ad -bc = 1\right\}_{/ A \sim - A}$$ 
which acts on $\Chat$ similarly by M\"obius transformations $z \mapsto (az + b)/(cz+d)$. In other words, the Jordan curve in a solution to the welding problem should be viewed as an equivalence class  in 
$$\PSL(2,\m C) \backslash \{\text{Jordan curves}\}$$
of the equivalence relation
\begin{equation}\label{eq:equivariant_curve}
\g \sim A \circ \g , \qquad \forall A \in \PSL(2,\m C). 
\end{equation}
We note that for a general circle homeomorphism, a solution to the welding problem may not exist, and if it exists, it may not be unique; see, e.g., \cite{Bishop_ann_welding} and the references therein. Characterizing all elements in $\Homeo_+(\RP^1)$ such that there is a unique solution is a challenging open question \cite{Bishop_CR_hard}. 
However, if $\weld$ is quasisymmetric, then the solution to the conformal welding problem exists and is unique  \cite{Beurling-Ahlfors}. The corresponding Jordan curves are called \emph{quasicircles}, and we denote the group of quasisymmetric homeomorphisms as 
\begin{align*}
    \QS(\RP^1): =  \Big \{ h\in \Homeo_+(\RP^1)   \,|\, & \exists M > 1, \, \forall \t \in \m R, \, \forall t \in (0,\pi), \,  \\
    & \frac1M \le \abs{\frac{\varphi(e^{\ii (\t + t)}) -\varphi( e^{\ii \t})}{\varphi(e^{\ii \t} )- \varphi(e^{\ii (\t-t)})}} \le M \Big\}
\end{align*}
where we have identified $\RP^1$ with the unit circle in $\m C$. In this article, all homeomorphisms of $\RP^1$ we consider are quasisymmetric.

We will explain that one should consider the welding homeomorphism $\weld_\g$ as a positive curve on the boundary of $\AdS^3$ --- Anti-de Sitter $3$-space --- so that the $\PSL(2,\m R) \times \PSL(2,\m R)$ action \eqref{eq:equivariant_welding} on $\weld_\g$ coincides with the orientation-preserving and time-preserving isometries of $\AdS^3$. This approach was taken in \cite{Mess07}, which gives a geometric interpretation of Thurston's earthquake theorem. See also \cite{Bonsante_Seppi,diaf2023antide}. 
Similarly, $\Chat$ is the conformal boundary of $\m H^3$ --- hyperbolic $3$-space --- so that the $\PSL(2,\m C)$ action \eqref{eq:equivariant_curve} coincides with the orientation-preserving isometries of $\m H^3$. 
Thus, conformal welding establishes a canonical correspondence between Jordan curves on $\partial_\infty \m H^3$ and positive curves on $\partial_\infty\!\AdS^3$ up to isometries. See Section~\ref{sec:H_AdS} for more details.

In recent years, the Loewner energy --- a quantity associated with each Jordan curve --- has caught a lot of attention for its close ties to random conformal geometry, Grunsky operators, universal Teichm\"uller space, and hyperbolic $3$-space \cite{W2,VW1,VW2,TT06,johansson2021strong,johansson2023coulomb,epstein_yilin,Bishop_WP,carfagnini_wang}.
Its definition first arises from the study of Schramm--Loewner evolutions \cite{W1, RW} and is invariant under the $\PSL(2,\m C)$ action \eqref{eq:equivariant_curve}. 
 Moreover, the Loewner energy coincides with a multiple of the universal Liouville action \cite{W2} introduced by Takhtajan and Teo \cite{TT06}, defined as a function on the associated welding homeomorphisms.  Takhtajan and Teo also showed that  the universal Liouville action is a K\"ahler potential of the unique homogeneous K\"ahler metric of the universal Teichm\"uller space $T(1) = \PSL(2, \m R) \backslash \QS (\m {RP}^1)$. (To be more precise, the K\"ahler potential is defined on the connected component $T_0(1)$ of $T(1)$ containing  the origin $[\Id_{\m {RP}^1}]$.)
Therefore,  it is natural to view the Loewner energy as a functional of the equivalence classes of Jordan curves 
and also of their corresponding classes of welding homeomorphisms.  See Section~\ref{sec:def_Loewner} for more details. See also \cite{W_AMS} for an expository article on the identity between Loewner energy and the universal Liouville action.


This paper aims to make the following observations about two optimization problems of the Loewner energy and compare them side-by-side. We will also provide a minimal background about hyperbolic space, AdS space, and the Loewner energy.

\begin{thm}[Optimizing the curve \cite{MRW1,MRW2}] \label{thm:opt_1}
    Let $z_1, \cdots, z_n$ be $n$ distinct points on $\Chat$, $\tau$  be a Jordan curve passing through these points. There exists a unique Jordan curve $\g$ minimizing the Loewner energy among all curves that are homotopic to $\tau$ relatively to $z_1, \cdots, z_n$. The welding homeomorphism $\weld_\g$ of $\g$ is $C^{1,1}$ and piecewise M\"obius, or equivalently, $\weld_\g$ is a continuously differentiable boundary of a pleated space-like geodesic plane in $\AdS^3$.
\end{thm}
This optimization problem was first studied in \cite{W1}, which pointed out that the minimizers have the geodesic property (see Definition~\ref{df:geod}). See also \cite{peltola_wang,Bonk_Eremenko} for similar geodesic properties in other setups. 
The existence and uniqueness of the minimizer are proved in \cite{MRW2} by showing that there exists a unique Jordan curve with the piecewise geodesic property. 
It was already noticed in \cite{MRW1} that piecewise geodesic Jordan curves have weldings that are piecewise M\"obius. 
We will recall some proofs in Section~\ref{sec:opt_curve}. The only addition here is the formulation of the minimizer in terms of the pleated planes in $\AdS^3$ space. By a \emph{pleated space-like geodesic 
 plane}, we mean a topological disk in $\AdS^3$ that is a union of finitely many subsets of space-like geodesic planes identified along complete geodesics. In particular, its intrinsic metric coincides with that of a space-like geodesic plane. See Figure~\ref{fig:pleated}.

 \begin{figure}
    \centering
    \includegraphics[width=.35\textwidth]{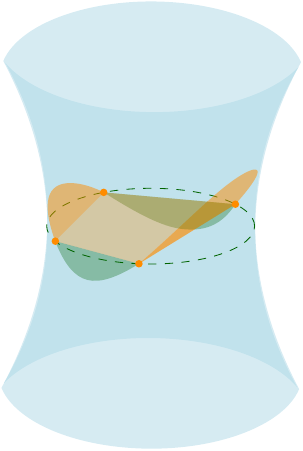}
    \caption{Illustration of a pleated plane in $\AdS^3$  with $C^1$ boundary where $n = 4$. In this case, the four points on $\partial_\infty\!\AdS^3$ are co-planar.}
    \label{fig:pleated}
\end{figure}

We define similarly \emph{pleated geodesic planes} in $\m H^3$ and can state the result for the second optimization problem.

\begin{thm}[Optimizing the welding] \label{thm:opt_2}
     Let $(x_1, \cdots, x_n)$ and $(y_1, \cdots, y_n)$  be two sets of $n$ distinct points of $\RP^1 \simeq S^1$ in counterclockwise order. If $\weld \in \Homeo_+(\RP^1)$ minimizes the Loewner energy among all homeomorphisms sending $x_k$ to $y_k$ for all $k$, then $\weld$ is the welding homeomorphism of a $C^{1,1}$ and piecewise circular curve. Equivalently, any positive curve in $\partial_\infty\!\AdS^3$ minimizing the Loewner energy among all of those passing through given $n$ points is the welding of a continuously differentiable boundary of a pleated geodesic plane in $\m H^3$.
\end{thm}
This result is proved in Section~\ref{sec:opt_welding}. We do not claim here the existence or uniqueness of the minimizer for this problem, although the existence should follow easily from a compactness argument, and there is strong evidence to believe in the uniqueness. These questions are under investigation and will appear elsewhere as the scope is quite different from this paper's.  A closely related problem of optimizing partial welding was investigated in \cite{mesikepp2022deterministic}, inspiring this work.

The similarity between the two optimization problems appears striking to the author, who finds it worthwhile to note and raise the question about the meaning of this symmetry. We will further comment on other similarities and discuss a few open questions at the end. 

Although we take a different approach here, we also mention the works \cite{Bishop_WP} and \cite{epstein_yilin} in the direction of establishing the correspondence between the Loewner energy of Jordan curves and geometric quantities in $\m H^3$. In \cite{Bishop_WP}, Bishop characterized qualitatively the curves with finite Loewner energy with those who bound minimal surfaces in $\m H^3$ with finite renormalized area.  In \cite{epstein_yilin}, the author and her collaborators established an identity between the Loewner energy of a Jordan curve and the renormalized volume of a hyperbolic $3$-manifold determined by the curve. 

\subsection*{Acknowledgment} The author thanks Fran\c cois Labourie, Tim Mesikepp, Andrea Seppi, and J\'er\'emy Toulisse for helpful discussions, and Steffen Rohde, Jinwoo Sung, Fredrik Viklund, and Catherine Wolfram for comments on an earlier version of this article.   

\section{Hyperbolic and Anti-de Sitter $3$-spaces} \label{sec:H_AdS}
We recall a minimal amount of basic facts about the hyperbolic and Anti-de Sitter spaces for readers unfamiliar with them. Readers who are familiar with these spaces may safely skip this section.  Most proofs are omitted except those that are more relevant to this work. The presentation on the $\AdS^3$ space follows closely to that of \cite{diaf2023antide}. 

\subsection{Hyperbolic $3$-space}
 The hyperboloid model of the hyperbolic $3$-space is given by
$$\m H^3 = \{X  \in \m R^{3,1} \,| \, q_{3,1} (X)= -1\}_{/X \sim - X}$$
where 
$$q_{3,1} ((x_1, x_2, x_3, x_4)) = x_1^2 + x_2^2 + x_3^2 - x_4^2$$
which induces a Riemannian metric on $\m H^3$.  We denote by $\brac{\cdot,\cdot}_{3,1}$ the bilinear form obtained by polarizing $q_{3,1}$.
Since $\{q_{3,1} (X)= -1\}$ has two connected components interchanged by $X \mapsto -X$, the quotient is simply choosing one connected component, or one can define $\m H^3: = \{q_{3,1} (X)= -1, x_4 > 0\}$. 

The hyperboloid model is isometric to the half-space model of the hyperbolic $3$-space
$$\m H^3 = \{(x, y, \eta) \,|\, \eta > 0, x, y \in \m R\}$$ endowed with the Riemannian metric $(\dd x^2 + \dd y^2 + \dd \eta^2)/\eta^2$. The  visual boundary is
$$\partial_\infty \m H^3 = \m C \cup \{\infty\} = \Chat $$
where we identified $(x,y,0) \in \partial_\infty \m H^3$ with $z \in x + \ii y \in \m C$. 
In the hyperboloid model, 
\begin{equation}\label{eq:boundary_H_hyperboloid}
\partial_\infty \m H^3  = \P \{X \in \m R^{3,1} \smallsetminus \{0\}\,|\, q_{3,1} (X) = 0\}
\end{equation}
where $\P$ denotes the projectivization.

\begin{remark}\label{rem:plane_H}
    In the half-space model, the geodesics in $\m H^3$ are the vertical rays and the half circles perpendicular to $\m C$. The totally geodesic planes are given by halfplanes perpendicular to $\m C$ and hemispheres bounded by circles in $\m C$. In the hyperboloid model, the totally geodesic planes are in the form of
\begin{equation}\label{eq:plane_H_hyperboloid}
    \{X \in \m H^3 \,|\, \brac{X,Y}_{3,1} = 0\},
\end{equation}
where $Y \in \m R^{3,1}$ is such that $q_{3,1} (Y) = 1$. 
\end{remark}

Every conformal automorphism of $\Chat$ extends to an orientation-preserving isometry of $\m H^3$, which can be expressed explicitly using quaternions. 
In fact, if we use $z +   \jj \eta \in \m C \oplus \jj \m R_+$ to parametrize $\m H^3$ and let $\begin{psmallmatrix}
    a & b \\ c & d
\end{psmallmatrix} \in \PSL(2,\m C)$, the corresponding M\"obius transformation $z \mapsto (az + b)/(cz+d)$ 
extends to the isometry of $\m H^3$
$$ Z \mapsto (aZ + b) (cZ+d)^{-1}, \quad \forall Z = z + \jj \eta $$
using the multiplication rules on quaternions, see \cite[Sec.\,2.1]{Ahlfors_Mobius}. Conversely, the boundary values of an orientation-preserving isometry of $\m H^3$ is a conformal automorphism of $\Chat$. Hence, we have
\begin{equation}
    \PSL(2,\m C) \simeq \Isom_+ (\m H^3)
\end{equation}
where $\Isom_+ (\m H^3)$ denotes the set of orientation-preserving isometries of $\m H^3$.
As a sanity check, we note that M\"obius transformations map circles in $\Chat$ to circles (we always consider lines as circles). This is consistent with the fact that isometries of $\m H^3$ map geodesic planes to geodesic planes.

\subsection{Anti-de Sitter $3$-space}
The $\AdS^3$ space is a Lorentzian manifold of signature $(2,1)$ defined similarly to the hyperboloid model of $\m H^3$ (that is why physicists often call $\m H^3$ the Euclidean\footnote{``Euclidean'' simply means that the signature is $(3,0)$.} $\AdS$ space), 
$$\AdS^3 = \{X  \in \m R^{2,2} \,| \, q_{2,2} (X)= -1\}_{/X = - X}$$
where 
$$q_{2,2} (x_1, x_2, x_3, x_4) = x_1^2 + x_2^2 - x_3^2 - x_4^2.$$

We can model $\m R^{2,2}$ using the space $M_2$ of  $2$ by $2$ real matrices and will only use this model in the sequel.  In fact, 
$$q(A) : = - \det (A), \qquad \forall A \in M_2$$
is a quadratic form of signature $(2,2)$. To see this, we note that the associated  bilinear form is 
$$\brac{A,B}_{2,2} = - \frac{1}{2} \tr (A \cdot \adj(B)) \quad
\text{ where }
\quad
\adj \begin{pmatrix}
    a & b \\ c & d
\end{pmatrix} =  \begin{pmatrix}
    d & -b \\ -c & a
\end{pmatrix},  $$
and an orthogonal basis is given by 
\begin{equation}\label{eq:basis}
    \Id = \begin{pmatrix}
    1 & 0 \\ 0 & 1
\end{pmatrix},\quad V =  \begin{pmatrix}
    0 & 1 \\ 1 & 0
\end{pmatrix}, \quad W =  \begin{pmatrix}
    1 & 0 \\ 0 & -1
\end{pmatrix}, \quad U =  \begin{pmatrix}
    0 & -1 \\ 1 & 0
\end{pmatrix}
\end{equation}
where we have $q (\Id) = q (U) = -1$ and $q (V) = q (W) = 1$. In this model,
$$\AdS^3 = \{ A \in M_2 \,|\, \det (A) = 1\}_{/A \sim - A} = \PSL(2,\m R).$$
It turns out that the group of orientation-preserving and time-preserving isometries of $\AdS^3$ is the group $\PSL(2,\m R) \times \PSL(2,\m R)$ acting as:
\begin{equation}\label{eq:isometry_AdS}
    (\a,\b) \cdot A  : = \a A \b^{-1}. 
\end{equation}
We refer the readers to \cite{diaf2023antide} for more details.

The visual boundary of $\AdS^3$ is identified in the same way as in \eqref{eq:boundary_H_hyperboloid}:
\begin{equation*}
\partial_\infty\!\AdS^3  = \P \{A \in M_2 \smallsetminus \{0\}\,|\, -\det (A) = 0\} = \P\{ A \in M_2 \,|\, A \text{ has rank }1\}
\end{equation*}
and we endow $\ad{\AdS^3} = \AdS^3 \cup \,\partial_\infty\!\AdS^3$ with the topology induced from $\P\!M_2$.
We parametrize $\partial_\infty\!\AdS^3$ by $\RP^1 \times \RP^1$ by the following map:
\begin{equation}\label{eq:delta}
    \delta \colon  \partial_\infty\!\AdS^3 \to \RP^1 \times \RP^1 \qquad [A] \mapsto (\Im (A), \Ker (A))
\end{equation}
where  $\Im (A)$ is the  image and $\Ker(A)$ is the  kernel of $A$. Here we identified $\RP^1$ with the space of one-dimensional subspaces of $\m R^2$. We will not explicitly mention  $\d$ when we identify $\partial_\infty\!\AdS^3$ with $\RP^1 \times \RP^1$.

There is a concrete description of the convergence to a point in $\partial_\infty\!\AdS^3$.

\begin{lem}[See {\cite[Lem.\,3.2.2]{Bonsante_Seppi}}] \label{lem:convergence}
      A sequence $(A_n)_{n \ge 0}$ in $\AdS^3$ converges to $A = (x,y) \in \partial_\infty\!\AdS^3$ if and only if $A_n (z) \to x$ and $A_n^{-1} (z) \to y$ for some $z \in \m H^2$ \textnormal(which also implies the limit holds for all $z \in \m H^2$\textnormal). Here $A_n (z)$ refers to the $\PSL(2,\m R)$ action by $A_n$ on $\m H^2 = \{z \in \m C\,|\, \Im (z) > 0\}$. 
\end{lem}

It is immediate that the action \eqref{eq:isometry_AdS} of $\PSL(2,\m R) \times \PSL(2,\m R)$  extends to $\partial_\infty\!\AdS^3$:
\begin{align}\label{eq:action_b_AdS} 
      (\alpha, \beta)\cdot(\Im (A), \Ker (A)) &=(\Im (\alpha A \beta^{-1}), \Ker (\alpha  A \beta^{-1})) \nonumber \\
      &= (\alpha \Im (A), \beta \Ker (A)).
\end{align} 
This implies immediately the following lemma which shows that it is very natural to encode welding homeomorphisms in $\partial_\infty\!\AdS^3$. 
\begin{lem}\label{lem:covariant_varphi}
  Let  $\weld: \RP^1 \to \RP^1$ be an orientation-preserving circle homeomorphism. The graph of $\weld$ 
$$\gr(\weld) := \{(x, \weld(x)) \,|\, x\in \RP^{1}\} \subset \partial_\infty\! \AdS^3$$  
is a positive curve in $\partial_\infty\!\AdS^3$. The action \eqref{eq:action_b_AdS} of orientation-preserving and time-preserving isometries of $\AdS^3$ on the graph coincides with the action \eqref{eq:equivariant_welding} on $\weld$:
$$(\a,\b) \cdot  \gr(\weld)  = \gr (\b \weld \a^{-1}).$$
\end{lem}
A closed curve in $\partial_\infty\!\AdS^3$ is said to be \emph{positive} if the totally geodesic plane passing through any three distinct points $(x_i, y_i)_{i = 1,2,3}$ on the curve is space-like (signature = $(2,0)$). As we will see in Lemma~\ref{lem:plane_mob}, the latter condition is equivalent to that there exists $\a \in \PSL(2,\m R)$ such that $\a (x_i) = y_i$ for all $i$, namely, $(x_1,x_2,x_3)$ is in the same orientation as $(y_1, y_2, y_3)$.

Now, we describe the space-like totally geodesic planes in $\AdS^3$. Similar to \eqref{eq:plane_H_hyperboloid}, they are given by 
\begin{equation}\label{eq:def_geo_plane}
P_{[\a]} = \{[A] \in \PSL(2,\m R) \,|\, \brac{\a,A}_{2,2} = 0\}
\end{equation}
where $\a \in M_2$ with $q (\a) < 0$ (to ensure that $P_{[\a]}$ is space-like). Since it defines the same plane if one replaces $\a$ by $c\a$ for some $c \neq 0$, we use the projective class $[\a]$ instead of $\a$.

\begin{ex}
    When $[\a] = [\Id]$, \eqref{eq:basis} shows that 
    $$P_{[\Id]} = \{A \in \Span\{U,V,W\} \,|\, q(A) = -1\}_{/A \sim -A} $$
   which is isometric to $\m H^2$ since $q (V) = q(W) = 1$ and $q (U) = -1$ and $U,V,W$ are orthogonal. Since $P_{[\a]} = \a P_{[\Id]} = (\a, \Id) \cdot P_{[\Id]}$ by \eqref{eq:isometry_AdS} and \eqref{eq:def_geo_plane}, $P_{[\a]}$ is also isometric to $\m H^2$.
\end{ex}
Let us recall the simple proof of the following statement. 
\begin{lem}\label{lem:plane_mob}
  Let $\a \in \PSL(2, \m R)$.   The boundary of $P_{[\a]}$ is $\gr (\a^{-1})$. In the latter, $\a$ is viewed as the M\"obius transformation of $\RP^1$ as in \eqref{eq:action_RP1}. In other words, the graphs of M\"obius transformations of $\RP^1$ are exactly the boundaries of space-like totally geodesic planes in $\AdS^3$.
\end{lem}
\begin{proof}
We show first that the boundary of $P_{[\Id]}$ is $\{(x,x) \,|\, x \in \RP^1\} = \gr (\Id_{\RP^1})$. 
For this,    we note that a matrix $A \in \Span \{U,V,W\}$ with $q(A) = -1$ if and only if $\tr(A) = 0$ and $\det (A) = 1$. 
These two conditions are equivalent to that the characteristic polynomial of $A$ is $X^2 + 1$. We have by Cayley--Hamilton theorem that these conditions are equivalent to $A^2 = -\Id$. 
So we have
    $$P_{[\Id]} = \{A \in M_2 \,|\, A^2 = - \Id\}/_{A \sim -A}.$$
   Now we describe the boundary of $P_{[\Id]}$. Let $A_n$ be a sequence in $P_{[\Id]} \subset \AdS^3$ converging to a boundary point $(x,y) \in \partial_\infty P_{[\Id]}$. From Lemma~\ref{lem:convergence}, we have for any $z \in \m H^2$,
    $$ A_n (z) \xrightarrow[]{n\to \infty} x, \quad A_n^{-1} (z) \xrightarrow[]{n\to \infty} y.$$
    On the other hand, $A_n^2 = - \Id$ implies that $A_n^{-1} (z) = - A_n (z) = A_n (z)$. We obtain $x = y$. To see that any point $(x,x)$ belongs to $\partial_\infty P_{[\Id]}$, we note that $\PSL(2,\m R)$ acts transitively by conjugation on $P_{[\Id]}$, which induces a transitive action on $\{(x,x) \,|\, x \in \RP^1\} = \gr (\Id_{\RP^1})$.

    Now let $\a \in \PSL(2,\m R)$, $P_{[\a]} =\a P_{[\Id]}  = (\a, \Id) \cdot P_{[\Id]}$. Lemma~\ref{lem:covariant_varphi} and the previous computation show that 
    $$\partial_\infty P_{[\a]} =(\a, \Id) \cdot  \partial_\infty P_{[\Id]} =  (\a, \Id) \cdot \gr(\Id_{\RP^1}) = \gr (\a^{-1})$$
    as claimed.
\end{proof}

\section{Optimizing the Loewner energy}
\subsection{Definition of the Loewner energy}\label{sec:def_Loewner}
The Loewner energy $I^L: \{\text{Jordan curves}\} \to [0,\infty]$ was introduced in \cite{RW} to measure the roundness of a Jordan curve. An important feature of the Loewner energy is its invariance under M\"obius transformations of $\Chat$. So $I^L$ is well defined on the equivalence classes $\PSL(2,\m C) \backslash \{\text{Jordan curves}\}$. 

We will give two equivalent definitions of Loewner energy. 
The first definition uses the \emph{Loewner chain}. From this definition, the claimed properties of the minimizers in Theorems~\ref{thm:opt_1} and \ref{thm:opt_2} follow immediately, as we will discuss in the next sections. The second definition connects to the Weil--Petersson metric on universal Teichm\"uller space and will only be discussed briefly. 

Let us first explain the notion of \emph{chordal Loewner driving function}.  We say that $\g$ is a {\em simple curve} in $(D;\bpt,\ept)$, if $D$ is a simply connected domain, $\bpt, \ept$ are two distinct prime ends of  $D$, and $\g$ has a continuous and injective parametrization $(0,T) \to D$ such that $\g(t) \to \bpt$ as $t \to 0$ and $\g(t) \to c \in D \cup \{\ept\}$ as $t \to T$. If $c = \ept$ then we say $\g$ is a {\em chord} in $(D;\bpt,\ept)$.
The driving function of a curve/chord in $(\domain;\bpt,\ept) $ is defined via the conformal map $(\domain;\bpt,\ept) \to (\m H^2; 0,\infty)$.
Therefore, it suffices to consider a simple curve/chord $\g$ in $(\m H^2; 0,\infty)$. 

We first parameterize $\g$ by the \emph{halfplane capacity}. More precisely, $\g$ is continuously parametrized by $[0,T)$, where $T \in (0, \infty]$ as above, 
in the way such that for all $t \in [0,T)$,
the unique conformal map $g_t$ from $\m H^2 \smallsetminus \gamma_{[0,t]}$ onto $\m H^2$ with the expansion at infinity $g_t (z)=  z + o(1)$ has the next order term $2t/z$, namely,
\begin{equation}\label{eq:hydro}
g_t (z)= z + \frac{2t}{ z} + o\left( \frac{1}{z}\right).
\end{equation}
The coefficient $2t$ is called the \emph{halfplane capacity} of $\g_{[0,t]}$.
It is easy to  show that $g_t$ can be continuously extended to the boundary point $\gamma_t$ and that the real-valued function $W_t := g_t ( \gamma_t)$ is continuous with $W_0 = 0$ (i.e., $W \in C^0[0,T)$). This function $W$ is called the \emph{driving function} of $\g$ and $2T$ the total capacity of $\g$.

\begin{remark}\label{rem:pathological}
There exist chords with total capacity $2T <\infty$. This happens only when $\g$ goes to infinity while staying close to the real line \cite[Thm.\,1]{LLN_capacity}. 
\end{remark}

The following properties of the Loewner driving function follow directly from the definition and are important.
\begin{itemize}
\item The imaginary axis $\ii \m R_+$ is a chord in $(\m H^2; 0, \infty)$ whose driving function is the zero function defined on $\m R_+$.
\item (\emph{Additivity}) Let $W_\cdot : [0,T) \to \m R$ be the Loewner driving function of a simple curve/chord $\g$ in $(\m H^2; 0,\infty)$. For every $0< s < T$, the driving function of $(g_s(\g[s,T]) - W_s)_{t \in [0,T-s)}$ is $ t \mapsto W_{s+t} - W_s$.
\item (\emph{Scaling property}) Fix $\l >0$, the driving function of the scaled and capacity-reparameterized simple curve/chord $(\l \g_{\l^{-2}t})_{t  \in [0, \l^2 T)}$ is  $t\mapsto \l W _{\l^{-2}t}$.
\end{itemize}

\begin{df}\label{df:chordalLE}
The \emph{chordal Loewner energy} of a simple curve/chord $\g$ in $(D;\bpt,\ept)$ is defined as the Dirichlet energy  of its driving function,
\begin{equation}\label{eq_LE}
    I^C_{D;\bpt, \ept} (\g) := I^C_{\m H; 0 ,\infty}(\varphi (\g)) : = I_{T} (W)  := \frac12 \int_{0}^{T} \left(\frac{\dd W_t}{\dd t}\right)^2 \dd t,
\end{equation} 
if $t \mapsto  W_t$ is absolutely continuous.  Otherwise, we set $I_{T} (W) = \infty$.
Here, $\varphi$ is any conformal map from $(D;\bpt,\ept)$ to $(\m H;0,\infty)$,
$W$ is the driving function of $\varphi (\g)$, $2T$ is the total capacity of $\varphi(\g)$. Moreover, if $\g$ is a chord and $I^C_{D;\bpt, \ept} (\g)<\infty$, then $T = \infty$, see, e.g., \cite[Thm.\,2.4]{yilin_survey}.
\end{df}
Note that the definition of $I^C_{D;\bpt, \ept} (\g)$ does not depend on the choice of $\varphi$. In fact, different choices differ only by post-composing by scaling. From the scaling property of the Loewner driving function, $W$ changes to $t\mapsto \l W_{\l^{-2} t}$ for some $\l > 0$, which has the same Dirichlet energy in its domain of definition.
\begin{remark}\label{rem:minimizer_geodesic}
The Loewner energy $I^C_{D;\bpt, \ept}$ is non-negative and only minimized  by the 
hyperbolic geodesic $\eta$ in $(\domain;\bpt, \ept)$ since the driving function of $\varphi(\eta)$ is the constant zero function.
We sometimes omit $a,b$ from $I^C_{D;\bpt, \ept} (\g)$ when $\g$ is a chord; in this case,  the two marked points are always taken to be the two end points of $\g$. Moreover, as shown in \cite{W1}, the Loewner energy of a chord is reversible, so we do not specify the order of the end points either. 
\end{remark}

Finally, we define the Loewner energy for a Jordan curve.
\begin{df}\label{df:Loewner_energy_jordan}
     Let $\g: [0,1] \to \Chat$ be a continuously parametrized Jordan curve with the marked point  (called \emph{root}) $\g (0) = \g(1)$. 
For every $\vare>0$, $\g [\vare, 1]$ is a chord connecting $\g(\vare)$ to $\g (1)$ in the simply connected domain $\Chat \smallsetminus \g[0, \vare]$.
The \emph{Loewner  energy} of $\g$ is defined as
$$I^L(\g): = \lim_{\vare \to 0} I^C_{\Chat \smallsetminus \g[0, \vare]} (\g[\vare, 1]) \in [0,\infty].$$
We define the \emph{Loewner energy of an arc} similarly.  Let $\g : [0,1] \to \Chat$ be a continuously parametrized simple arc with the marked point $\g(0)$ ($\neq \g(1)$). 
For every $\vare>0$, $\g [\vare, 1]$ is a simple curve in the simply connected domain $\Chat \smallsetminus \g[0, \vare]$ with finite capacity.
The Loewner  energy of $\g$ is 
$$I^A(\g): = \lim_{\vare \to 0} I^C_{\Chat \smallsetminus \g[0, \vare]; \g(\vare), \g(0)} (\g[\vare, 1])  \in [0,\infty].$$
\end{df}

     Although not obvious, it is true that the Loewner energy of a Jordan curve does not depend on any particular parametrization of $\g$, in particular, neither on the orientation nor the root, as first proved in \cite{W1,RW}. The following theorem explains these symmetries by providing an equivalent expression of the Loewner energy, which does not involve any parametrization of the Jordan curve. 
     
 As the Loewner energies $I^L$ and $I^A$ are defined using the Loewner driving function, it is immediate that they are $\PSL(2,\m C)$-invariant.  In the next theorem, we assume $\g$ is a Jordan curve that does not pass through $\infty$ without loss of generality.

\begin{thm}[See {\cite[Thm.\,1.4]{W2}}]\label{thm:intro_equiv_energy_WP} \label{thm:main}
Let $\O$ \textnormal(resp., $\O^*$\textnormal) denote the component of $\Chat \smallsetminus \g$ which does not contain $\infty$ \textnormal(resp., which contains $\infty$\textnormal) and $f$ \textnormal(resp., $g$\textnormal) be a conformal map from the unit disk $\m D = \{z \in \Chat  \colon |z| < 1 \}$ onto $\O$ \textnormal(resp., from $\m D^* = \{z \in \Chat \colon |z| >1 \}$ onto $\O^*$\textnormal). We assume further that $g (\infty) = \infty$.
The Loewner energy of $\gamma$ satisfies
       \begin{equation} \label{eq_disk_energy}
   I^L(\gamma) = \frac{1}{\pi} \int_{\m D} \abs{\frac{f''}{f'}}^2 \dd^2 z + \frac{1}{\pi} \int_{\m D^*} \abs{\frac{g''}{g'}}^2 \dd^2 z +4 \log \abs{\frac{f'(0)}{g'(\infty)}}, 
 \end{equation}
 where $g'(\infty):=\lim_{z\to \infty} g'(z)$ and $\dd^2 z$ is the Euclidean area measure.
\end{thm}

The right-hand side of \eqref{eq_disk_energy} was introduced by Takhtajan and Teo in \cite{TT06} as a functional associated with the welding $h = g^{-1} \circ f$, and is called the \emph{universal Liouville action}.
Jordan curves with finite Loewner energy are also called \emph{Weil--Petersson quasicircles} (they are a special type of quasicircles), originally defined as the Jordan curves for which $f''/f' \in L^2 (\m D)$. The following result characterizes Weil--Petersson quasicircles via their welding homeomorphisms. 

\begin{thm}[See \cite{shen13}]\label{thm:shen}
    A Jordan curve is Weil--Petersson if and only if its welding homeomorphism $\weld$ satisfies that
    $\log \weld'$ is in the Sobolev $H^{1/2} (S^1)$ space. We call such $\weld$ a \emph{Weil--Petersson circle homeomorphism} and write
    $\WP(\RP^1)$ for the set of all Weil--Petersson circle homeomorphisms.
\end{thm}

\begin{lem}\label{lem:add_energy} Let us collect a few properties which follow directly from Definition~\ref{df:Loewner_energy_jordan} of the Loewner energy:
\begin{enumerate}
    \item   $I^L$ and $I^A$ are  $\PSL(2,\m C)$-invariant.
    \item  If $\g_1$ is an arc and $\g_2$ is a chord in $\Chat \smallsetminus \g_1$ connecting the two end points of $\g_1$, then the additivity property of the Loewner driving function shows
       \begin{equation}\label{eq:add_energy}
       I^L (\g_1 \cup \g_2) = I^A (\g_1) + I_{\Chat \smallsetminus \g_1}^C (\g_2). 
       \end{equation}
       \item  We have $I^A (\g_1) = I^L (\g_1 \cup \eta)$ where $\eta$ is the hyperbolic geodesic in $\Chat \smallsetminus \g_1$ connecting the two end points of $\g_1$.
       \item \label{it:arc_min} $I^L (\g) = 0$ if and only if $\g$ is a circle 
       in $\Chat$, and this implies $I^A (\g_1) = 0$ if and only if $\g_1$ is a circular arc. 
\end{enumerate}
\end{lem}

\subsection{Optimizing the curve} \label{sec:opt_curve}
The following question was studied in \cite{MRW1,MRW2}, and we recall the results. Given $n$ distinct points $\mathbf{z} = (z_1, \ldots, z_n)$ on $\Chat$ and a Jordan curve $\tau$ passing through these $n$ points in order, we let 
$\mc L (\mathbf{z}; \tau)$ be the set of all Jordan curves passing through $z_1, \ldots, z_n$ and homotopic to $\tau$ relative to $z_1, \ldots, z_n$ (this means the points $z_1, \ldots, z_n$ remain fixed along the continuous deformation of the Jordan curve). In particular, if $\g \in \mc L (\mathbf{z}; \tau)$, then $\g$ traverses the $n$ points in the same order as $\tau$ (the converse is not true, see Figure~\ref{fig:two_curve}). 

\begin{figure}[ht]
    \centering
    \includegraphics[width=0.5\textwidth]{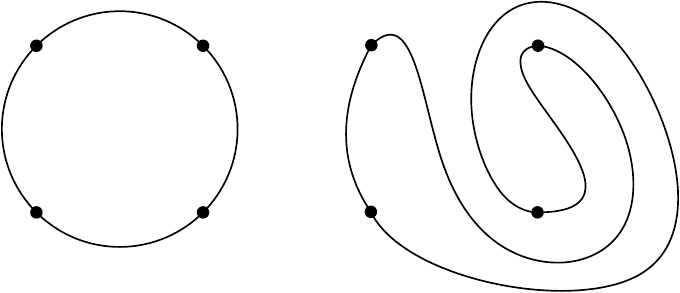}
    \caption{Example of two non-homotopic curves passing through the same four points in the same order.}
    \label{fig:two_curve}
\end{figure}

\begin{df}\label{df:geod}
    We say that a Jordan curve $\g$ in $\mc L (\mathbf{z}; \tau)$ has the \emph{geodesic property} if for all $k \in \{1, \ldots, n\}$, $\g_k$ is the hyperbolic geodesic in the simply connected domain $\O_k : = \Chat \smallsetminus (\g \smallsetminus \g_k)$, where $\g_k$ denotes the part of $\g$ between $z_{k}$ to $z_{k+1}$ (where $z_{n+1} = z_1$).
\end{df}

\begin{lem}\label{lem:minimizer_geodesic}
    If $\g$ minimizes  $I^L$ in $\mc L (\mathbf{z}; \tau)$, then $\g$ has the geodesic property.
\end{lem}
\begin{proof}
    From Equation~\eqref{eq:add_energy}, we know that for all $k \in \{1, \ldots, n\}$, 
    \begin{equation}\label{eq:add_opt_1}
    I^L(\g) = I^A(\g \smallsetminus \g_k) + I^C_{\O_k} (\g_k).
    \end{equation}
    If $\g_k$ is not the hyperbolic geodesic $\eta_k$ in $\O_k$, then replacing $\g_k$ by $\eta_k$ we obtain another Jordan curve in $\mc L (\mathbf{z}; \tau)$ but with smaller Loewner energy by Remark~\ref{rem:minimizer_geodesic}. This contradicts the assumption that $\g$ minimizes the Loewner energy in $\mc L (\mathbf{z}; \tau)$.
\end{proof}
\begin{lem}[See also \cite{MRW1,SWW}]\label{lem:pw_mob}
    If $\g \in \mc L (\mathbf{z}; \tau)$ has the geodesic property, then any welding homeomorphism $\weld$ of $\g$ is piecewise M\"obius \textnormal(in $\PSL(2,\m R)$\textnormal). If, in addition, $\g$ is Weil--Petersson, then $\weld$ is $C^{1,1}$.
\end{lem}
\begin{proof}
    We use the notation in Definition~\ref{df:geod}.  Let $\O$ and $\O^*$ be the two connected components of $\Chat \smallsetminus \g$ on the left and right of $\g$, respectively. Let $f : \m H^2 \to \O$ and $g : \m H^{2,*} \to \O^*$ be any conformal maps. Then by definition $\weld : = g^{-1} \circ f$ is a welding homeomorphism of $\g$. 

    Let $k \in \{0,\ldots, n-1\}$ and $\varphi_k$ be a conformal map from $\m C \smallsetminus \m R_-$ onto $\O_k$ sending respectively $0$ and $\infty$ to  $z_k$ and $z_{k+1}$.  Since $\g_k$ is the hyperbolic geodesic in $\O_k$ connecting $z_k$ and $z_{k+1}$,  $\varphi_k ( \m R_+) =\g_k$. In other words, $\varphi_k |_{\m H^2}$ is a conformal map from $\m H^2$ onto $\O$.  Therefore, there exists $\a_k \in \PSL(2,\m R)$ such that $f = \varphi_k |_{\m H^2} \circ \a_k^{-1}$. Similarly, there exists $\b_k \in \PSL(2,\m R)$ such that $g =  \varphi_k |_{\m H^{2,*}}\circ \b_k^{-1}$. 
    Since $\varphi_k$ is continuous across $\m R_+$,
    $$\weld (x) = g^{-1} \circ f (x) = \b_k \circ \a_k^{-1} (x) \qquad \forall x \in \a_k (\m R_+) = f^{-1} (\g_k).$$
    We obtain that $\weld$ is piecewise $\PSL(2,\m R)$.
In this case, there are two possibilities: 
\begin{itemize}
    \item If $\weld$ is continuously differentiable, then its derivative is also Lipschitz, namely, $\weld$ is $C^{1,1}$. Theorem~\ref{thm:shen} implies that $\weld$ is Weil--Petersson. 
    \item If $\weld$ is not continuously differentiable, then its derivative has a discontinuity, and $\weld$ is not Weil--Petersson by Theorem~\ref{thm:shen}. 
\end{itemize}
This dichotomy proves the second claim.
\end{proof}

\begin{thm}[See \cite{MRW1,MRW2}]\label{thm:opt_1_unique}
    There is a unique minimizer of $I^L$ in $\mc L (\mathbf{z}; \tau)$, which is also the unique Jordan curve in $\mc L (\mathbf{z}; \tau)$ that has the geodesic property and finite Loewner energy.
\end{thm}
The existence of a minimizer is shown by taking an energy minimizing sequence in $\mc L (\mathbf{z}; \tau)$ and a compactness argument allows us to extract a converging subsequence. 
The uniqueness of the Jordan curve in $\mc L (\mathbf{z}; \tau)$ that has the geodesic property and finite Loewner energy is more subtle and proved in \cite{MRW2}.

Putting these results together, we obtain Theorem~\ref{thm:opt_1}.
\begin{proof}[Proof of Theorem~\ref{thm:opt_1}]
    From Theorem~\ref{thm:opt_1_unique} we know there is a unique energy minimizer $\g$ in $\mc L (\mathbf{z}; \tau)$. Lemma~\ref{lem:minimizer_geodesic} shows that $\g$ has the geodesic property. 
 Lemma~\ref{lem:pw_mob} shows that any welding homeomorphism $\weld$ of $\g$ is $C^{1,1}$ and piecewise $\PSL(2,\m R)$. 
    
    The graph $\gr (\weld)$ is a closed positive curve in $\partial_\infty\!\AdS^3$ as we explained in Lemma~\ref{lem:covariant_varphi}. Each piece of $\weld$ that is M\"obius is a sub-interval of the boundary of a space-like geodesic plane in $\AdS^3$ by Lemma~\ref{lem:plane_mob}. Therefore, $\gr (\weld)$ is the boundary of a pleated space-like plane that is also $C^{1,1}$.
    
    To give an explicit construction of the pleated plane, we consider an arbitrary topological triangulation with vertices $(x_1, y_1), \ldots, (x_n, y_n)$ that has the topology of a disk. We replace each topological triangle with the ideal triangle in the space-like geodesic plane passing through the three vertices. These ideal triangles and the half-planes bounded by each piece of $h$ glue together to a space-like pleated plane in $\AdS^3$.  
\end{proof}

\subsection{Optimizing the welding}  \label{sec:opt_welding}

Now, we consider the second optimization problem. Let $\mathbf{x} = (x_1, \cdots, x_n)$ and $\mathbf{y} = (y_1, \cdots, y_n)$  be two sets of $n$ distinct points of $\RP^1 \simeq S^1$ in counterclockwise order.
We write
$$\Phi_{\mathbf{x,y}} : = \{\weld  \in \Homeo_+ (\RP^1) \,|\, \weld (x_k) = y_k,\, \forall k = 1, \ldots, n\}.$$
In terms of the $\AdS^3$ space, the graphs of elements in $\Phi_{\mathbf{x,y}}$ is the set of all positive curves in $\partial_\infty\!\AdS^3$ that pass through the $n$ points $(x_k,y_k)$, $k = 1, \ldots, n$.

We consider Loewner energy $I^L$ as a functional on $\Homeo_+(\RP^1)$ by assigning 
$I^L(\weld)$ to be the Loewner energy of any Weil--Petersson quasicircle that is the solution to the welding problem of $\weld$, when $\weld$ is Weil--Petersson (there is no ambiguity since $I^L$ is $\PSL(2,\m C)$-invariant), and $I^L(\weld) = \infty$ if $\weld$ is not Weil--Petersson.

\begin{df}
    We say that a Jordan curve in $\Chat$ is \emph{piecewise circular} if it is a concatenation of circular arcs.
\end{df}

The following result is the analog of Lemma~\ref{lem:minimizer_geodesic}. 

\begin{lem}\label{lem:pw_circular}
    If a circle homeomorphism minimizes $I^L$ in $\Phi_{\mathbf{x,y}}$, then it is the welding homeomorphism of a $C^{1,1}$ piecewise circular Jordan curve.
\end{lem}
\begin{proof}
For any $\weld \in \Phi_{\mathbf{x,y}}$, we write $\weld_k$ for the restriction of $\weld$ on $[x_k, x_{k+1}]$. 
We let $\g$ be a Jordan curve with conformal welding $\weld$, namely, $\weld = g^{-1} \circ f$ where $f,g$ are conformal maps associated with $\g$ as in the introduction. We write $\g_k = f([x_k, x_{k+1}])$.

From Equation~\ref{eq:add_energy}, we know that for all $k \in \{1,\ldots, n\}$,
\begin{equation}\label{eq:add_opt_2}
I^L(\g) = I^A (\g_k) + I^C_{\Chat 
\smallsetminus \g_k} (\g \smallsetminus \g_k). 
\end{equation}
The second term on the right-hand side equals $I^C_{\m H^2; 0,\infty} (\varphi_k (\g \smallsetminus \g_k))$ where $\varphi_k$ is a conformal map $\Chat 
\smallsetminus \g_k \to \m H^2$ sending the two end points of $\g_k$ to $0,\infty$.
This chordal energy depends only on $\{\weld_i\,|\,i \neq k\}$. Therefore, Property \ref{it:arc_min} of Lemma~\ref{lem:add_energy} implies that if we replace $\g_k$ by a circular arc, adjust $\varphi_k$ and $\g \smallsetminus \g_k$ such that $\varphi_k (\g \smallsetminus \g_k)$ remains unchanged, then the energy of $\g$ decreases unless $\g_k$ is already a circular arc. 
More explicitly, we note that $\varphi_k$ conjugates the welding homeomorphism $h|_{[x_k, x_{k+1}]}$ to a welding map $\m R_- \to \m R_+$. Replacing this welding map by $x \mapsto -x$, the map $s: z \mapsto z^2$ welds $\m R_+$ to $\m R_-$, the Jordan curve $\tilde \g : = s ( \varphi_k (\g \smallsetminus \g_k) )\cup \m R_+$ is a Jordan curve with smaller energy.  Moreover, the curve $\tilde \g$ has a welding in $\Phi_{\mathbf{x,y}}$ given by 
$$\tilde h (x)= \tilde g^{-1} \circ \tilde f (x)=  \begin{cases}
   g^{-1} \circ \varphi_k^{-1} \circ \varphi_k \circ f (x) =  h (x),\qquad &  x \notin (x_k, x_{k+1})\\
 g^{-1} \circ \varphi_k^{-1} (- \varphi_k \circ f (x) ), & x \in (x_k, x_{k+1}).
\end{cases}$$
See Figure~\ref{fig:opt_weld}. This shows that $\g_k$ has to be a circular arc (which is the image of $\m R_+$ under $\PSL(2,\m C)$). Hence, any minimizer of the Loewner energy in $\Phi_{\mathbf{x,y}}$ is the conformal welding of a piecewise circular Jordan curve.
\begin{figure}[ht]
    \centering
    \includegraphics[width=0.65\textwidth]{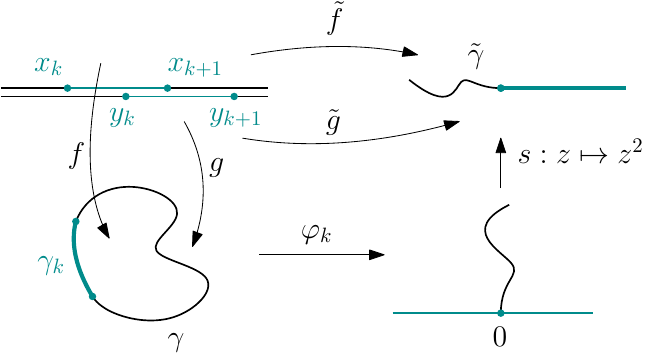}
    \caption{Conformal maps in the proof of Lemma~\ref{lem:pw_circular}.}
    \label{fig:opt_weld}
\end{figure}

To see that the corresponding Jordan curve is $C^{1,1}$, it suffices to show that the curve has a continuous derivative: circular arcs meet at $\pi$ angles. In fact, otherwise, the curve would have infinite Loewner energy; see, for example, \cite[Sec.\,2.1]{RW} or using the fact that Weil--Petersson quasicircles are asymptotically conformal (which do not allow corners).
\end{proof}

\begin{proof}[Proof of Theorem~\ref{thm:opt_2}]
    This follows immediately from Lemma~\ref{lem:pw_circular} and the fact that the totally geodesic planes in $\m H^3$ are bounded by circles in $\partial_\infty \m H^3 = \Chat$ (Remark~\ref{rem:plane_H}). The explicit construction of a pleated plane bounded by the piecewise circular curve is exactly the same as that in the proof of Theorem~\ref{thm:opt_1}.
\end{proof}

\section{Comments and open questions}

We observe the similarities between the proofs of Theorems~\ref{thm:opt_1} and \ref{thm:opt_2}. For instance, we used the same additivity property \eqref{eq:add_energy} which follows immediately from the Loewner-chain definition of Loewner energy. 
For any minimizer in Theorem~\ref{thm:opt_1}, the chordal energy in the decomposition~\eqref{eq:add_opt_1} has to vanish, whereas for a minimizer in Theorem~\ref{thm:opt_2}, the arc energy has to vanish in  \eqref{eq:add_opt_2}. However, to apply the additivity to the decomposition of the Jordan curve for all $k \in \{1,\ldots, n\}$ in Equations \eqref{eq:add_opt_1} and \eqref{eq:add_opt_2}, we used the root-independence of the Loewner energy which is more apparent from the expression of the universal Liouville action \eqref{eq_disk_energy}.

Despite the similarity of the statement and the proof between the two optimization problems, there are still some differences between the statements. For instance, we do not have an expression of the Loewner energy directly in terms of the positive curve in $\partial_\infty \!\AdS^3$ or the welding. The energy for a circle homeomorphism is currently defined by first solving the welding problem.
\begin{qn}
   A recent work \cite{epstein_yilin} expressed the Loewner energy of a Jordan curve in terms of geometric quantities in $\m H^3$. Can we express $I^L (\weld)$ solely in terms of the circle homeomorphism $\weld$ or geometric quantities in $\AdS^3$? 
\end{qn}

There are a few other similarities between the two optimization problems. 

\noindent \textbf{Loop Loewner driving function.} We can define the Loewner driving function for Jordan curves, which may be viewed as a consistent family of chordal Loewner driving functions described in Section~\ref{sec:def_Loewner}. See, e.g., \cite[Sec.\,3]{sung2023quasiconformal} for details on the loop driving function. 
In general, the Loewner driving function of a Jordan curve (including its domain of definition) depends on the choice of the orientation of the curve, the \emph{root} $\g(0)$, another marked point $\g(1/2)$ on the curve, and a conformal map $\Chat \smallsetminus \g [0,1/2] \to \m H$ sending the two marked points to $0, \infty$, where $\g [0, 1/2]$ is the sub-arc of $\g$ from $\g (0)$ to $\g(1/2)$.

For a Weil--Petersson quasicircle, the driving function is a continuous function $\m R \to \m R$, and its Dirichlet energy coincides with its Loewner energy defined in  Definition~\ref{df:Loewner_energy_jordan}. Moreover, the minimizers of our optimization problems have the following properties which follow from the definition of the loop Loewner driving function:

\begin{itemize}
    \item Let $\g$ be the unique minimizer of the Loewner energy in $\mc L (\mathbf{z}; \tau)$. Let $\l_{\cdot} : \m R \to \m R$ be any driving function of $\g$ while rooted at any point. Then there is $T \in \m R$ such that $\l_t$ is constant for all $t \ge T$.
    \item Let $\g$ be a Jordan curve whose welding is an minimizer in $\Phi_{\mathbf{x,y}}$ (if it exists). We know from Theorem~\ref{thm:opt_2} that $\g$ is piecewise circular. 
    Let $\l_{\cdot}: \m R \to \m R$ be any driving function of $\g$ while rooted at any point. Then there is $T \in \m R$ such that $\l_t$ is constant for all $t \le T$.
\end{itemize}
It seems reasonable to ask the following question, given these observations.
\begin{qn}
    Can we relate the driving function of a minimizer in Theorem~\ref{thm:opt_1} to the time reversal of that of a minimizer in Theorem~\ref{thm:opt_2}?
\end{qn}

\noindent \textbf{Schwarzian derivatives.} We recall that the \emph{Schwarzian derivative} of a conformal map $f$ is given by 
$$\mc S [f] := \frac{f'''}{f'} - \frac32 \left(\frac{f''}{f'}\right)^2.$$
Schwarzian derivatives satisfy the chain rule
$$\mc S[f\circ g] = \mc S[f] \circ g \,(g')^2 + \mc S[g] $$
and $\mc S[A] \equiv 0$ if and only if $A \in \PSL(2,\m C)$ is a M\"obius transformation. Schwarzian derivatives define quadratic differentials $\mc Q[f]: = \mc S[f](z) \,\dd z^2$ on $\Chat$ or $\m H$, so that 
\begin{equation} \label{eq:quad_dif}
\mc Q[f \circ A] = A^* \mc Q[f], \qquad \forall A \in \PSL(2,\m C).
\end{equation}
\begin{itemize}
    \item It was already noted in \cite{MRW1} that the conformal maps $f : \m H^2 \to \O$, $g:\m H^{2,*} \to \O^*$ associated with the minimizer in $\mc L (\mathbf{z}; \tau)$ satisfy
    $$\mc S[f^{-1}] (z) = \mc S[g^{-1}] (z), \qquad \forall z \in \g \smallsetminus \{z_1, \ldots, z_n\}.$$
    This follows from the chain rule of Schwarzian derivative and Lemma~\ref{lem:pw_mob}.
    Moreover, \cite{MRW1} shows that the union of $\mc S[f^{-1}]$ and $\mc S[g^{-1}]$ defines a meromorphic function on $\Chat$ with simple poles at $ \{z_1, \ldots, z_n\}$.   The work \cite{MRW2} shows that the residue at $z_k$ of this meromorphic function can be expressed in terms of the minimal energy of the curves in $\mc L (\mathbf{z}; \tau)$ and its derivatives with respect to moving the point $z_k$.
    
    If we consider the quadratic differentials $\mc Q[f^{-1}]$ and $\mc Q [g^{-1}]$ to be associated with the Jordan curve, then from \eqref{eq:quad_dif},  we see that they project to the equivalence classes of  $\PSL(2,\m C) \backslash \{\text{Jordan curves}\}$. 
  
    \item Minimizers in $\Phi_{\mathbf{x,y}}$  give rise to a similar Schwarzian derivative. More precisely, let $\g$ be a minimizer and $f,g$ be the associated conformal maps as above. Since $f$ maps the upper halfplane onto a domain bounded by circular arcs, Schwarz reflection also implies that $\mc S [f]$ extends to a meromorphic function on $\m C$ with poles at $\mathbf {x} = \{x_1, \ldots, x_n\}$. It is also not hard to see that the poles are also simple. Similarly, $\mc S[g]$ extends to another meromorphic function on $\Chat$ with simple poles at $\mathbf {y} = \{y_1, \ldots, y_n\}$. We do not know how these two meromorphic functions are related. 

    If we consider the quadratic differentials $\mc Q[f]$ and $\mc Q [g]$ to be associated with the welding $h = g^{-1} \circ f$, they are well defined on the equivalence classes of weldings with respect to \eqref{eq:equivariant_welding} which is induced by the equivalence relations $f \sim f \circ \a$ and $g \sim g \circ \b$ for $\a, \b \in \PSL(2,\m R)$.
    
\end{itemize}
    \begin{qn}
        Can we relate the residues of  $\mc S [f]$, where $f$ is a conformal map from $\m H^2$ onto a domain bounded by a piecewise circular curve, to the Loewner energy of the curve? 
    \end{qn}

\noindent \textbf{Wick rotation} From the hyperboloid models of $\m H^3$ and $\AdS^3$, one sees that these two spaces are related by simply turning one space-like direction in $\m R^{3,1}$ into a time-like direction in $\m R^{2,2}$, namely, the Wick rotation by replacing $x_3$ by $\ii x_3$ where $\ii^2 = -1$. From the discussion above, although it is far from being clear to us, we speculate that many objects or concepts come in pairs and may possibly be considered as the Wick rotation of each other. We leave it as an open question to give a precise interpretation as a Wick rotation to the pairs in rows of the following Table~\ref{table}.

\begin{longtable}{ l  c  l }
 \toprule
    \multicolumn{1}{c}{\bf  Hyperbolic space $\m H^3$}  &     \multicolumn{1}{c}{$\xleftrightarrow{\text{W.R.}}$} &  \multicolumn{1}{c}{\bf Anti-de Sitter space  $\AdS^3$} \\ \hline 
     Jordan curves in $\Chat = \partial_\infty \m H^3$ &  $\xleftrightarrow{}$ & Circle homeomorphisms and \\
   & & positive curves in  $\partial_\infty \!\AdS^3$\\
      \hline
       Loewner energy of a Jordan curve & $\xleftrightarrow{}$ &  Loewner energy of a welding \\      
      &&homeomorphism\\
       \hline 
      (1) Energy minimizer in $\mc L (\mathbf{z}; \tau)$& & (1') Energy minimizer in $\Phi_{\mathbf{x,y}}$  
       \\
       Shear along geodesics in $\m H^2 \subset \Chat$ &$\xleftrightarrow{}$&  for some $\mathbf{x}$ and $\mathbf{y}$\\
    (deforming $\partial \m H^2$ to a piecewise  &  &  (unknown characterization of the \\
    geodesic Jordan curve) &   & minimizer in terms of welding)\\ 

       \hline
      (2) (Boundary of) pleated geodesic & $\xleftrightarrow{}$ & (2') (Boundary of) pleated space-  \\
     planes   in $\m H^3$  &&  like geodesic plane in $\AdS^3$ \\\hline
     Quadratic differentials   & $\xleftrightarrow{}$ &  Quadratic differentials \\
    $\mc Q[f^{-1}]$ and $\mc Q[g^{-1}]$  && $\mc Q[f]$ and $\mc Q[g]$ \\

      \hline
      Driving function of a Jordan curve & $\xleftrightarrow{}$& Time reversal of the driving \\
      && function (+ some modifications)   \\
  \bottomrule 
  \caption{Conjectural Wick-rotation type correspondence between Jordan curves and circle homeomorphisms.} \label{table}
\end{longtable}

Theorems~\ref{thm:opt_1} and \ref{thm:opt_2} show that the welding of the item (1) in the table is (2'), and the welding of (2) is (1'). It seems to us that the following question is one natural way to give a concrete wish list about the Wick rotation.
\begin{qn}
    Is there a consistent way to make sense of Wick rotations in the rows of the table, such that the Loewner energy is invariant under the Wick rotation and the welding operation commutes with the Wick rotation? By ``commuting'' we mean that the chain of operations
    $$\text{(1)} \xrightarrow[]{\text{W.R.}} \text{(1')} \xrightarrow []{\text{solve welding problem}}\text{(2)}  \xrightarrow[]{\text{W.R.}}  \text{(2')}$$
    coincides with 
       $$\text{(1)}  \xrightarrow []{\text{welding homeomorphism}} \text{(2')}.$$
\end{qn}
\bibliographystyle{abbrv}
\bibliography{ref}
\end{document}